\documentclass[11pt]{amsart}
\usepackage[a4paper,margin=1in]{geometry}
\usepackage{amsmath,amssymb,amsthm,mathtools}
\usepackage[colorlinks=true,citecolor=blue,linkcolor=blue,urlcolor=blue]{hyperref}

\allowdisplaybreaks
\numberwithin{equation}{section}

\newtheorem{theorem}{Theorem}[section]
\newtheorem{proposition}[theorem]{Proposition}
\newtheorem{lemma}[theorem]{Lemma}

\theoremstyle{remark}
\newtheorem{remark}[theorem]{Remark}

\newcommand{\C}{\mathbb C}
\newcommand{\R}{\mathbb R}
\newcommand{\Sph}{\mathbb S}
\newcommand{\Ric}{\operatorname{Ric}}
\newcommand{\diver}{\operatorname{div}}
\newcommand{\tr}{\operatorname{tr}}

\title[Lagrangian balls with conformal Maslov form]
{Rigidity of Lagrangian submanifolds with conformal Maslov form and Legendrian capillary boundary}

\author{Dong Gao}
\address{School of Science, Beijing University of Civil Engineering and Architecture,
Beijing 102616, P.R. China}
\email{gaodong@bucea.edu.cn}

\author{Yong Luo}
\address{Mathematical Science Research Center of Mathematics, Chongqing University of
Technology, Chongqing 400054, P.R. China}
\email{yongluo-math@cqut.edu.cn}

\author{Hui Ma}
\address{Department of Mathematical Sciences, Tsinghua University, Beijing 100084,
P.R. China}
\email{ma-h@tsinghua.edu.cn}

\author{Jiabin Yin}
\address{School of Mathematics and Statistics, Xinyang Normal University,
Xinyang 464000, P.R. China}
\email{jiabinyin@126.com}


\begin{document}

\begin{abstract}
We classify smoothly immersed Lagrangian $n$-balls, $n\geq 2$, in the unit ball of $\mathbb{C}^n$ with conformal Maslov form and Legendrian capillary boundary. The image of every such immersion is either  an equatorial
Lagrangian disk or is contained in a Whitney sphere centered at the origin.
In dimension two, this confirms a conjecture of Li, Wang and Weng [Sci. China Math. 2021].
\end{abstract}

\maketitle

\section{Introduction}

Free boundary problems ask how an interior geometric equation
interacts with the position of the boundary. For a minimal disk in
the Euclidean three-ball, Nitsche proved that the free boundary
condition forces the disk to be equatorial \cite{Nitsche1985}.
Fraser and Schoen extended disk rigidity to balls in space forms
of arbitrary ambient dimension \cite{FraserSchoen2015}. A Lagrangian
submanifold in a complex ball has an additional constraint: its
boundary may be required to be Legendrian in the contact sphere.
This leads to a different boundary condition from the usual
orthogonality condition for a free boundary submanifold.

Let $\iota:M^n\to\C^n$ be a Lagrangian immersion. We take its mean
curvature vector $H$ to be the trace of the second fundamental form.
The one-form dual to $JH$ is the Maslov form; it is closed in
$\C^n$. We say that the Maslov form is conformal when
\begin{equation}\label{eq:intro-conformal}
\nabla_X(JH)=\frac{1}{n}\diver(JH)X.
\end{equation}
The closedness of the Maslov form explains why conformality takes
the stronger form \eqref{eq:intro-conformal}. In dimension two,
Castro and Urbano \cite{CastroUrbano1993} associated a holomorphic
cubic differential with this condition. They classified compact
orientable examples, including the Whitney sphere and a family of
Lagrangian tori. In arbitrary dimension, Ros and Urbano
\cite{RosUrbano1998} studied the closed conformal field $JH$.
They characterized the Whitney sphere by its second fundamental
form and proved rigidity for compact Lagrangian submanifolds
with vanishing first Betti number.

Several later results clarify the role of curvature and the
trace-free second fundamental form. Castro, Montealegre and
Urbano \cite{CastroMontealegreUrbano2001} studied closed conformal
fields on Lagrangian submanifolds of complex space forms.
Chao and Dong \cite{ChaoDong2012} proved a curvature pinching
result in this setting. Cao and Zhao \cite{CaoZhao2021} obtained
gap theorems, and Zhao and Cao \cite{ZhaoCao2022} gave further
Whitney sphere characterizations. Zhang \cite{Zhang2021}, Luo
and Yin \cite{LuoYin2022}, and Luo and Zhang \cite{LuoZhang2023}
established related energy gap results. These theorems concern
closed submanifolds or impose additional curvature or energy
conditions. They do not directly address a Lagrangian ball with
capillary boundary.

Li, Wang and Weng \cite{LiWangWeng2021} introduced the Legendrian
capillary boundary condition in the unit ball. Along such a
boundary, $N=\iota$ is the outward normal to the unit sphere and
$\nu$ is the outward conormal in $M$. The boundary is Legendrian,
and the two normals satisfy
\begin{equation}\label{eq:intro-angle}
\nu=\sin\theta\,N+\cos\theta\,JN
\end{equation}
for a constant $\theta\in[0,\pi]$ on each boundary component
\cite{LiWangWeng2021}. The case $\theta=\pi/2$ is the Legendrian
free boundary condition. %The capillary condition permits other
%angles and admits nonminimal examples.

Li, Wang and Weng proved that a minimal Lagrangian disk with
Legendrian capillary boundary is an equatorial disk
\cite[Theorem~1.1]{LiWangWeng2021}. Their Conjecture 
\cite[Conjecture~2.15]{LiWangWeng2021} asks
whether a disk with conformal Maslov form must be planar or a
Whitney cap. Luo and Sun \cite{LuoSun2021} established equatorial rigidity
for minimal Lagrangian surfaces with Legendrian free boundary and classified the minimal Lagrangian
capillary annuli in $\C^2$. Gaia \cite{Gaia2025} proved a related
free boundary result for Hamiltonian stationary Lagrangian disks
with Legendrian boundary, assuming under suitable regularity, continuity of the Lagrangian angle, and a localization property. More recently, Gao,
Ma and Yao \cite{GaoMaYao2026} proved equatorial rigidity for
minimal and, more generally, self-similar Lagrangian immersions
with connected Legendrian capillary boundary. Gao, Luo, Ma and
Yin \cite{GaoLuoMaYinSelfSimilar2026} classified the self-similar
case according to the number of boundary components. The present
problem allows a nonminimal mean curvature vector that satisfies
\eqref{eq:intro-conformal} without satisfying a self-similar
equation.

For $r>0$ and $c\in\C^n$, write the Whitney immersion as
\begin{equation}\label{eq:whitney}
W_{r,c}(x_1,\ldots,x_{n+1})
=c+\frac{r(1+ix_{n+1})}{1+x_{n+1}^2}
(x_1,\ldots,x_n),\qquad (x_1,\ldots,x_{n+1})\in\Sph^n.
\end{equation}
For $c=0$ and $r>1$, the preimage
$W_{r,0}^{-1}(\overline{\mathbb B}^{2n})\subset \Sph^n$
has two connected components, each diffeomorphic to the closed $n$-ball.
The restriction of $W_{r,0}$ to either component is a Lagrangian immersion
with Legendrian capillary boundary.
These are the higher-dimensional analogues of the examples in
\cite[Proposition~2.10]{LiWangWeng2021}. Their boundary latitudes are
$x_{n+1}=\pm\sqrt{(r^2-1)/(r^2+1)}$. The case $r=1$ also gives an example under our angle convention: the restrictions of $W_{1,0}$ to
$\{x_{n+1}\le0\}$ and $\{x_{n+1}\ge0\}$ have Legendrian capillary
boundary with $\theta=0$ and $\theta=\pi$, respectively.

We classify Lagrangian balls with these boundary conditions. The
result proves the disk conjecture of Li, Wang and Weng and gives
the corresponding statement in every dimension $n\geq2$.

\begin{theorem}\label{thm:main}
Let $\iota:M^n\to\overline{\mathbb B}^{2n}$ be a smooth Lagrangian
immersion with conformal Maslov form, where $n\geq2$ and $M^n$ is diffeomorphic to the closed
$n$-ball. Suppose that $\iota(M^\circ)\subset\mathbb B^{2n}$,
$\iota(\partial M)\subset\mathbb S^{2n-1}$, and the boundary is Legendrian capillary. Then $\iota$ is
an equatorial Lagrangian disk, or its image is contained in a Whitney
sphere centered at the origin.
\end{theorem}

The difficulty is already visible for surfaces. Li, Wang and Weng
observed that the holomorphic differential used in the minimal
disk proof retains an uncontrolled boundary term when $H$ is
nonzero \cite[Section~2.6]{LiWangWeng2021}. In higher dimensions,
the surface argument has no scalar holomorphic counterpart.
We first show that $JH$ is normal to the boundary. For surfaces,
we use a local boundary ODE. In higher dimensions, we compare two
Ricci curvature identities on the boundary. A harmonic one-form
then gives the Whitney second fundamental form in the interior.

Section~2 collects the geometric and boundary identities. Section~3
treats surfaces, and Section~4 treats higher dimensions. Section~5
proves the classification and identifies the center of the Whitney
sphere.

\section{Preliminaries}

We collect the tensor identities, closed conformal field properties,
and boundary formulas used in the proof. Our curvature convention is
\begin{equation}\label{eq:curvature-convention}
 R(X,Y)Z=\nabla_X\nabla_YZ-\nabla_Y\nabla_XZ
          -\nabla_{[X,Y]}Z.
\end{equation}
If $S$ is a symmetric covariant tensor, its covariant derivative is
called symmetric when it is symmetric in all its indices.

\subsection{The Lagrangian cubic tensor}

For a Lagrangian immersion, $J$ maps $TM$ isometrically to its normal
bundle. Let $D$ be the Euclidean connection and $\nabla$ the induced
connection on $M$. Write
\begin{equation}\label{eq:gauss-formula}
        D_XY=\nabla_XY+h(X,Y),
\end{equation}
where $h$ is the second fundamental form. We use the trace convention
\begin{equation}\label{eq:trace-H}
        H=\sum_{i=1}^n h(e_i,e_i)
\end{equation}
for the mean curvature vector. The Lagrangian cubic tensor
\begin{equation}\label{eq:cubic-A}
        A(X,Y,Z)=\langle h(X,Y),JZ\rangle
\end{equation}
is symmetric. The Euclidean Codazzi equation and
$\nabla_X^\perp(JY)=J\nabla_XY$ show that $\nabla A$ is symmetric as
a four-tensor.
Moreover,
\begin{equation}\label{eq:trace-A}
 \sum_{i=1}^n A(e_i,e_i,X)
 =\langle H,JX\rangle
 =-\langle JH,X\rangle.
\end{equation}

For a one-form $\eta$, set
\begin{equation}\label{eq:symmetric-product}
 \begin{split}
 (\eta\mathbin{\odot}g)(X,Y,Z)
  ={}&\eta(X)g(Y,Z)+\eta(Y)g(X,Z)\\
    &+\eta(Z)g(X,Y).
 \end{split}
\end{equation}
Put $V=JH$, and define $\alpha_H:=V^\flat=(JH)^\flat$ to be the Maslov form in our convention. Thus $\alpha_H(X)=\langle JH,X\rangle$. The tensor
\begin{equation}\label{eq:tracefree-A}
        \mathring A=A+\frac{1}{n+2}(\alpha_H\mathbin{\odot}g)
\end{equation}
is trace-free.  We have
\begin{lemma}\label{Symmetric}
If the Maslov form is conformal, so that
\begin{equation}\label{eq:closed-conformal}
        \nabla_XV=\rho X,
        \qquad \rho=\frac1n\diver V,
\end{equation}
then $\nabla\mathring A$ is symmetric.
\end{lemma}
\begin{proof}
The Codazzi equation gives full symmetry of $\nabla A$. Equation
\eqref{eq:closed-conformal} gives
$(\nabla_X\alpha_H)(Y)=\rho g(X,Y)$. Since $\nabla g=0$,
\[
\nabla_X(\alpha_H\mathbin{\odot}g)(Y,Z,Q)
=\rho\bigl\{g(X,Y)g(Z,Q)+g(X,Z)g(Y,Q)
             +g(X,Q)g(Y,Z)\bigr\}.
\]
The right-hand side is symmetric in $X,Y,Z,Q$, as required.
\end{proof}
\subsection{Closed conformal vector fields}

We first record the zero-set and level-set properties of a closed
conformal vector field.

\begin{lemma}\label{lem:zero-set}
Let $(P^n,g)$, $n\geq2$, be a connected smooth Riemannian manifold,
possibly with smooth boundary. Suppose that a vector field $X$ and a function $\sigma$, smooth up to the boundary, satisfy
\begin{equation}\label{eq:X-conformal}
        \nabla_YX=\sigma Y
\end{equation}
for every vector field $Y$.  If $X\not\equiv0$, then $\sigma(p)\neq0$ at every zero $p$ of $X$. Consequently, all zeros of $X$ are nondegenerate and isolated relative to $P$.  In particular, if $X$ vanishes on a nonempty smooth
hypersurface, possibly contained in $\partial P$, then $X\equiv0$ on $P$.
\end{lemma}

\begin{proof}
For the standard zero-set statement, compare
\cite[Lemma~1(3)]{RosUrbano1998}. We give a direct proof that also covers dimension two and manifolds with smooth boundary, and establishes the additional uniqueness statements in the lemma.

Equation \eqref{eq:X-conformal} gives
\begin{equation}\label{eq:curvature-X}
        R(Y,Z)X=Y(\sigma)Z-Z(\sigma)Y.
\end{equation}
Taking a trace and using \eqref{eq:curvature-convention}, we obtain
\begin{equation}\label{eq:gradient-sigma}
        \Ric(X,Y)=-(n-1)Y(\sigma).
\end{equation}
Along a smooth curve $\gamma$, the pair $(X,\sigma)$ therefore
satisfies
\begin{equation}\label{eq:zero-ode}
 \nabla_{\dot\gamma}X=\sigma\dot\gamma,
 \qquad
 \frac{d}{dt}\sigma
 =-\frac{1}{n-1}\Ric(X,\dot\gamma).
\end{equation}
This is a homogeneous linear system for $(X,\sigma)$. If both
vanish at one point, uniqueness along piecewise smooth curves gives
$X\equiv0$ on the connected manifold, including its boundary.

Suppose now that $X$ is nontrivial and $X(p)=0$.  Then
$\sigma(p)\neq0$.  Hence
\begin{equation*}
        \nabla X(p)=\sigma(p)\operatorname{Id}
\end{equation*}
is invertible, and $p$ is an isolated nondegenerate zero.  Finally, if
$X$ vanishes on a hypersurface, its derivative in every tangential
direction of this hypersurface vanishes.  Equation \eqref{eq:X-conformal} gives
$\sigma=0$ on that hypersurface, and the first part of the proof shows
$X\equiv0$.
\end{proof}

In the compact setting, Ros--Urbano obtained the spherical structure of
the leaves orthogonal to the conformal vector field in the proof of
\cite[Theorem~3]{RosUrbano1998}.  We record
below a direct version that only requires compactness of the regular
leaves and connectedness of the regular set.

\begin{lemma}\label{lem:round-leaves}
Let $(P^n,g)$, $n\geq 2$, be a connected smooth Riemannian manifold without boundary, and
let $X\not\equiv 0$ be a smooth vector field satisfying
\eqref{eq:X-conformal}. Suppose that $X$ has a zero, that
$P^n\setminus Z(X)$ is connected, and that every regular leaf of
$X^\perp$ is compact, where $Z(X)=\{p\in P:X(p)=0\}$. Then every such leaf, with its induced metric, is isometric to a round $(n-1)$-sphere.
\end{lemma}

\begin{proof}
Set  $w=|X|$ and $E=X/w$ on $P^n\setminus Z(X)$. If $Y\perp E$, then
\[
Y(w^2)
=2\langle\nabla_YX,X\rangle
=2\sigma\langle Y,X\rangle
=0.
\]
Using $X=wE$ and \eqref{eq:X-conformal}, we obtain
\begin{equation}\label{eq:E-geometry}
        \nabla_EE=0,
        \qquad
        E(w)=\sigma,
        \qquad
        \nabla_YE=\frac{\sigma}{w}Y
        \quad \text{for }Y\perp E.
\end{equation}
 These identities also imply that $dE^\flat=0$. Thus $E^\perp$ is integrable, and Cartan's formula gives
\[
\mathcal L_E(E^\flat)=\iota_E dE^\flat+d(E^\flat(E))=0.
\]
Hence the flow of $E$ preserves $E^\perp$. The integral curves of $E$ are geodesics and the leaves of $E^\perp$ are totally umbilical.

Let $s$ be the flow parameter of $E$. Since $E$ is unit and its flow preserves $E^\perp$, using the flow to identify nearby leaf patches, we write 
\(
g=ds^2+g_s
\).
Here $g_s$ is the induced metric on the corresponding leaf patch,
pulled back by the flow. Since $w$ is constant along each connected
leaf, it depends only on $s$ in these coordinates.
Equation \eqref{eq:E-geometry} gives
\[
        \frac{d}{ds}g_s
        =2\frac{\sigma}{w}g_s
        =2\frac{\partial_s w}{w}g_s.
\]
Hence
\begin{equation}\label{eq:normalized-leaf-metric}
        \frac{d}{ds}\bigl(w^{-2}g_s\bigr)=0,
\end{equation}
and therefore, in a flow neighborhood,
\begin{equation}\label{eq:local-warped}
        g=ds^2+w(s)^2g_0
\end{equation}
for a fixed metric $g_0$ on one of the leaves.

Now let $p\in Z(X)$.  By Lemma \ref{lem:zero-set},
$\sigma(p)\neq0$. In a sufficiently small geodesic normal
neighborhood of $p$, put $r=d(p,\cdot)$.
Along each unit-speed radial geodesic $\gamma_\vartheta$,
\[
\nabla_{\dot\gamma_\vartheta}X
=\sigma\dot\gamma_\vartheta,
\qquad X(p)=0.
\]
Consequently,
\[
X=f(r,\vartheta)\partial_r,
\qquad
f(r,\vartheta)
=\int_0^r \sigma(\gamma_\vartheta(t))\,dt.
\]
After shrinking the neighborhood, $\sigma$ has the fixed
sign of $\sigma(p)$. Hence $f$ has the same sign, and
\[
E=\operatorname{sgn}(\sigma(p))\,\partial_r.
\]
The geodesic spheres centered at $p$ are therefore integral
hypersurfaces of $E^\perp$. Since $Y(w)=0$ for $Y\perp E$
and these spheres are connected, $w=w(r)$ and $f=f(r)$.
Moreover,
\[
\frac{w(r)}{r}\longrightarrow|\sigma(p)|
\quad\text{as }r\to0.
\]
Each such sphere is an entire leaf because it is compact and is
both open and closed in the connected leaf containing it.

Write the metric in geodesic polar coordinates as
\(
        g=dr^2+g_r.
\)
Since $\partial_s=E=\operatorname{sgn}(\sigma(p))\partial_r$,
\eqref{eq:normalized-leaf-metric} gives
\[
        \frac{d}{dr}\bigl(w^{-2}g_r\bigr)=0.
\]
On the other hand, as $r\to0$,
\[
        r^{-2}g_r\longrightarrow g_{\Sph^{n-1}},
        \qquad
        \frac{w(r)}{r}\longrightarrow|\sigma(p)|.
\]
It follows that
\begin{equation}\label{eq:round-near-zero}
        g_r
        =\frac{w(r)^2}{\sigma(p)^2}g_{\Sph^{n-1}}.
\end{equation}
Thus all sufficiently small regular leaves around $p$ are round
$(n-1)$-spheres up to constant scaling.

It remains to propagate roundness to every regular leaf. Let
$\mathcal F$ be a compact regular leaf. 
Since $E$ is transverse to $\mathcal F$, its flow gives,
for sufficiently small $\varepsilon>0$, a product
neighborhood
\(
(-\varepsilon,\varepsilon)\times\mathcal F
\).
Each image $\Phi_s(\mathcal F)$ is an entire leaf:
it is a compact connected integral hypersurface and hence
is both open and closed in the connected leaf containing it.
By \eqref{eq:normalized-leaf-metric},
\[
\Phi_s^*g_{\Phi_s(\mathcal F)}
=\frac{w(s)^2}{w(0)^2}\,g_{\mathcal F},
\]
where $w(s)$ denotes the constant value of $w$ on
$\Phi_s(\mathcal F)$.
Thus roundness of one leaf in this neighborhood implies
roundness of every leaf in it.

Now fix an arbitrary regular leaf $\mathcal F$ and choose a regular
leaf $\mathcal F_0$ sufficiently close to $p$. By
\eqref{eq:round-near-zero}, $\mathcal F_0$ is round up to scaling.
Since $P^n\setminus Z(X)$ is connected, there is a curve in
$P^n\setminus Z(X)$ joining $\mathcal F_0$ to $\mathcal F$. Its compact
image can be covered by finitely many overlapping flow neighborhoods
of the type described above. Passing successively through these
neighborhoods, the roundness of $\mathcal F_0$ propagates to
$\mathcal F$. Hence every regular leaf of $X^\perp$ is, up to a
constant scaling, a round $(n-1)$-sphere.
\end{proof}

\subsection{Vanishing of Codazzi tensors}

\begin{lemma}\label{lem:codazzi-sphere}
Let $(\Sph^n_K,g)$, $n\geq2$, be a round sphere of constant sectional curvature $K>0$, and let $S$ be a symmetric trace-free covariant
$q$-tensor, $q\geq 2$. If $\nabla S$ is symmetric, then $S=0$. 
\end{lemma} 
 
\begin{proof} 
Since $\nabla S$ is 
symmetric and $S$ is trace-free, 
\begin{equation}\label{eq:T-divfree} 
        \begin{aligned} 
\nabla^{a}S_{a i_2\cdots i_q} 
&= g^{ab}\nabla_b S_{a i_2\cdots i_q} \\ 
&= g^{ab}\nabla_{i_2} S_{ab i_3\cdots i_q} \\ 
&= \nabla_{i_2}\left(g^{ab}S_{ab i_3\cdots i_q}\right) \\ 
&= 0. 
\end{aligned} 
\end{equation} 
At a point where the frame is normal, use the constant-curvature 
identity 
\begin{equation*} 
 R_{b i a}{}^c 
 =K\bigl(\delta_{ia}\delta_b^c 
          -\delta_{ba}\delta_i^c\bigr). 
\end{equation*} 
The Codazzi equation and \eqref{eq:T-divfree} give 
\begin{align*} 
\nabla^a\nabla_aS_{i_1\cdots i_q} 
&=g^{ab}\nabla_b\nabla_{i_1}S_{a i_2\cdots i_q}\\ 
 &=g^{ab}([\nabla_b,\nabla_{i_1}]S)_{a i_2\cdots i_q}\\ 
 &=-g^{ab}R_{b i_1 a}{}^cS_{c i_2\cdots i_q}\\ 
 &\quad- 
   \sum_{r=2}^q g^{ab}R_{b i_1 i_r}{}^c 
   S_{a i_2\cdots i_{r-1}c i_{r+1}\cdots i_q}. 
\end{align*} 
The first curvature term contributes $(n-1)K S_{i_1\cdots i_q}$. 
For each $r\geq2$, the part containing 
$\delta_{i_1i_r}$ is a trace of $S$ and vanishes, while the remaining 
part is $K S_{i_1\cdots i_q}$.  Hence 
\begin{equation}\label{eq:rough-T} 
        \nabla^a\nabla_aS=(n+q-2)KS. 
\end{equation} 
Taking the $L^2$ inner product with $S$ and integrating gives 
\begin{equation*} 
 -\int_{\mathbb S^n_K}|\nabla S|^2 
 =(n+q-2)K\int_{\mathbb S^n_K}|S|^2. 
\end{equation*} 
Both sides can have the stated signs only when $S=0$. 
\end{proof}

The next two lemmas turn tensor information into vanishing results.

\begin{lemma}\label{lem:harmonic-contraction}
Let $S$ be a symmetric trace-free cubic tensor on a Riemannian manifold $(P^n,g)$ such that
$\nabla S$ is symmetric.  Suppose that a vector field $X$ satisfies
$\nabla_YX=\sigma Y$ for every vector field $Y$. Then the one-form $\eta=S(X,X,\cdot)$
is closed and coclosed.
\end{lemma}

\begin{proof}
For vector fields $Y$ and $Z$, 
\begin{equation}\label{eq:derivative-eta} 
 (\nabla_Y\eta)(Z) 
 =(\nabla_YS)(X,X,Z)+2\sigma S(Y,X,Z). 
\end{equation} 
The right-hand side of \eqref{eq:derivative-eta} is symmetric in $Y$ and $Z$, so $d\eta=0$. 

Let $\{e_i\}$ be a local orthonormal frame.  The symmetry of $\nabla S$ 
and the trace-free property give
\begin{equation*} 
 \sum_i(\nabla_{e_i}S)(X,X,e_i)=0, 
 \qquad 
 \sum_iS(e_i,X,e_i)=0. 
\end{equation*} 
Taking the trace of \eqref{eq:derivative-eta}, we obtain 
\begin{align*} 
-\delta\eta 
&=\sum_i (\nabla_{e_i}\eta)(e_i)\\ 
&=\sum_i (\nabla_{e_i}S)(X,X,e_i) 
  +2\sigma\sum_i S(e_i,X,e_i)\\ 
&=0. 
\end{align*} 
\end{proof}
\begin{remark}
In the proof of \cite[Theorem~3]{RosUrbano1998}, Ros--Urbano
introduced the vector field
\[
Z_{\rm RU}
=
J h(J\bar H,J\bar H)
-\frac{3n}{n+2}|\bar H|^2J\bar H,
\]
where $\bar H=\frac1n\tr h$ is the normalized mean curvature vector, and proved that $Z_{\rm RU}$ is closed and
divergence-free.  With our convention $H=\tr h$, put $V=JH$ and
$\eta=\mathring A(V,V,\cdot)$.  A direct computation gives
\[
\eta^\sharp
=
-Jh(V,V)+\frac{3}{n+2}|V|^2V
=
-n^2 Z_{\rm RU}.
\]
%Thus Lemma~\ref{lem:harmonic-contraction} may be viewed as the tensorial formulation of the
%harmonicity calculation used by Ros--Urbano.
\end{remark}
\begin{lemma}\label{lem:leaf-reduction}
Let $U$ be an open subset of a Riemannian manifold $(P^n,g)$, $n\geq2$.
Suppose that $X$ is nowhere zero on $U$ and satisfies
$\nabla_YX=\sigma Y$. Assume that every leaf of $X^\perp$ in $U$, with induced metric, is isometric to a round $(n-1)$-sphere.
Let $S$ be a symmetric trace-free cubic tensor on $U$ with symmetric
covariant derivative. If
\begin{equation}\label{eq:two-X-zero}
        S(X,X,\cdot)=0,
\end{equation}
then $S=0$ on $U$.
\end{lemma}

\begin{proof}
Set $w=|X|$, $E=X/w$. Equation \eqref{eq:two-X-zero} gives $S(E,E,\cdot)=0$.

If $n=2$, choose a local orthonormal frame $(E,T)$.
Trace-freeness gives 
\[
S(E,T,T)=-S(E,E,E)=0, \qquad 
S(T,T,T)=-S(E,E,T)=0.
\]
Together with $S(E,E,\cdot)=0$, these identities give $S=0$.

Assume $n\geq3$, and fix a leaf $\mathcal{L}$ of $X^\perp$ in $U$. Define
\begin{equation*}
        B(Y,Z)=S(E,Y,Z), \qquad Y,Z\in T\mathcal L.
\end{equation*}
For an orthonormal frame $\{e_1, \cdots, e_{n-1}\}$ tangent to $\mathcal L$,  trace-freeness of $S$ gives
\begin{equation*}
        \tr_\mathcal{L}B=\sum_{\alpha=1}^{n-1}S(E,e_\alpha,e_\alpha)=-S(E,E,E)=0.
\end{equation*}
For $Y,Z\in T\mathcal{L}$, \eqref{eq:E-geometry} gives
\[
\nabla_YE=\frac{\sigma}{w}Y,
\qquad
\nabla_YZ=\nabla_Y^\mathcal{L}Z-\frac{\sigma}{w}
\langle Y,Z\rangle E.
\]
Using these identities and $S(E,E,\cdot)=0$, we obtain
\begin{equation}\label{eq:B-derivative}
 (\nabla_YS)(E,Z,Q)
 =(\nabla_Y^\mathcal{L}B)(Z,Q)-\frac{\sigma}{w}S(Y,Z,Q)
\end{equation}
for $Y,Z,Q\in T\mathcal{L}$.  The symmetry of $\nabla S$ and $S$ therefore implies that
$\nabla^\mathcal{L}B$ is symmetric. Since $\mathcal L$ is a round sphere of dimensional $n-1\geq 2$,  Lemma \ref{lem:codazzi-sphere}, with $q=2$, gives $B=0$.

Now let $S^\mathcal{L}=S|_{T\mathcal{L}^3}$.  Its trace satisfies
\[
(\tr_{\mathcal L}S^{\mathcal L})(Z)=-S(E,E,Z)=0.
\]
Since $B=0$, we have
\begin{equation*}
        (\nabla_YS)(Z,Q,R)=(\nabla_Y^\mathcal{L}S^\mathcal{L})(Z,Q,R)
\end{equation*}
for $Y,Z,Q,R\in T\mathcal L$.  
Thus $\nabla^\mathcal{L}S^\mathcal{L}$ is symmetric.
Lemma \ref{lem:codazzi-sphere}, with $q=3$, gives $S^\mathcal{L}=0$. 
The identities $S(E,E,\cdot)=0$, $B=0$, and $S^{\mathcal L}=0$ account for all components of $S$ along $\mathcal L$.
Since the leaf is arbitrary, $S=0$ on $U$.
\end{proof}

\subsection{Legendrian capillary boundary identities}
 Let $\iota:M^n\to\overline{\mathbb B}^{2n}$, $n\geq 2$, be a smooth Lagrangian immersion with Legendrian capillary boundary $\Sigma=\partial M$. We use the notation and contact angle convention in \eqref{eq:intro-angle}.

The following boundary identities follow from \cite[Lemma~2.1]{LuoSun2021}; see also \cite[Proposition~2.11]{LiWangWeng2021} for the curvature-line characterization in the surface case. We include a direct proof in the notation used below.

\begin{lemma}\label{lem:boundary-identities}
For $U,Z\in T\Sigma$, one has
\begin{equation}\label{eq:boundary-basic}
\nabla_U\nu=\sin\theta U,
\qquad
A(U,Z,\nu)=\cos\theta g(U,Z),
\qquad
A(U,\nu,\nu)=0.
\end{equation}
\end{lemma}

\begin{proof}
Since $N=\iota$ on the unit sphere,
\begin{equation*}
D_UN=U,
\qquad D_U(JN)=JU.
\end{equation*}
The angle is constant, so differentiating \eqref{eq:intro-angle} gives
\begin{equation}\label{eq:Dnu}
D_U\nu=\sin\theta U+\cos\theta JU.
\end{equation}
As the immersion is Lagrangian, taking the tangent and normal components of \eqref{eq:Dnu}, we get
\begin{equation*}
\nabla_U\nu=\sin\theta U,
\qquad h(U,\nu)=\cos\theta JU.
\end{equation*}
Consequently,
\begin{equation*}
A(U,\nu,Z)=\langle h(U,\nu),JZ\rangle
=\cos\theta g(U,Z),
\qquad
A(U,\nu,\nu)=0.
\end{equation*}
The symmetry of $A$ completes the proof.
\end{proof}

When $n=2$, let $T$ be a unit tangent field along $\Sigma$ and define
the signed geodesic curvature by $\nabla_TT=k_g\nu$. Metric compatibility gives
\[
k_g=\langle \nabla_TT,\nu\rangle=-\langle T,\nabla_T\nu\rangle=-\sin\theta.
\]
Thus
\begin{equation}\label{eq:surface-boundary-identities}
A(T,\nu,\nu)=0,
\qquad
A(T,T,\nu)=\cos\theta,
\qquad
k_g=-\sin\theta.
\end{equation}

Assume in addition that the Maslov form is conformal, so that $V=JH$ satisfies \eqref{eq:closed-conformal}. Along the boundary,
write
\begin{equation}\label{eq:V-decomposition}
V=W+a\nu,
\qquad W\in T\Sigma,
\qquad a=\alpha_H(\nu)=\langle V,\nu\rangle.
\end{equation}
Here $a$ is a smooth function on $\Sigma$.

For $U,Z\in T\Sigma$, the first identity in \eqref{eq:boundary-basic} gives
\begin{equation}\label{eq:boundary-second-fundamental}
\nabla_UZ=\nabla_U^\Sigma Z-\sin\theta g(U,Z)\nu.
\end{equation}
Differentiating \eqref{eq:V-decomposition} and using this identity, 
we obtain
\[
\begin{aligned}
\nabla_UV&=\nabla_UW+U(a)\nu+a\nabla_U\nu\\
&=\nabla_U^\Sigma W+a\sin\theta U+(U(a)-\sin\theta\langle W,U\rangle)\nu.
\end{aligned}
\]
Comparing the tangential and conormal components with $\nabla_U V=\rho U$ yields
\begin{equation}\label{eq:Wa-equations}
\nabla_U^\Sigma W=(\rho-a\sin\theta)U,
\qquad
U(a)=\sin\theta\langle W,U\rangle.
\end{equation}
Thus $W$ is a closed conformal vector field on $\Sigma$ with
\begin{equation}\label{eq:boundary-sigma}
\sigma=\rho-a\sin\theta,
\end{equation}
and
\begin{equation}\label{eq:W-conformal}
\nabla_U^\Sigma W=\sigma U.
\end{equation}

\section{Boundary rigidity for surfaces}\label{sec:surface-case}

Throughout this section, we assume the hypotheses of Theorem~\ref{thm:main} with $n=2$ and write $M$ for $M^2$. Our aim is to  show that $W=0$ along $\partial M$.
 In an oriented unit frame $(T,\nu)$
along the boundary, write
\begin{equation}\label{eq:surface-W-scalar}
        W=wT,
        \qquad w=\langle JH,T\rangle.
\end{equation}
At a boundary point where $H=0$, one already has $W=0$. It therefore suffices to work near boundary points where $H\neq0$.

In a local isothermal coordinate $z=x+iy$, write 
$\partial_z=\frac12(\partial_x-i\partial_y)$ and extend tensors complex linearly.
Define
\begin{equation}\label{eq:surface-cubic}
        \Phi=A(\partial_z,\partial_z,\partial_z)\,dz^3.
\end{equation}
For a Lagrangian surface with conformal Maslov form, $\Phi$ is
holomorphic \cite[Corollary~1]{CastroUrbano1993};
see also \cite[Proposition~2.14]{LiWangWeng2021}. 

\subsection{One-sided Castro--Urbano coordinates}

The interior normal form is due to Castro and Urbano
\cite[Section~3]{CastroUrbano1993}. We give a proof valid up to a smooth boundary, without assuming
that the boundary is a coordinate line.

\begin{lemma}\label{lem:surface-normal-form}
Let $p\in\partial M$ with $H(p)\neq 0$.  There exist a relative neighborhood $U$ of $p$, a complex coordinate
$z=x+iy$ that is conformal on $U\cap M^\circ$ and smooth up to
$U\cap\partial M$, a smooth real function $u$ and constants
\begin{equation*}
        \alpha,\lambda\in\R,
        \qquad \mu\geq0,
\end{equation*}
such that
\begin{equation}\label{eq:surface-normal-metric}
        g=e^{2u}|dz|^2,
        \qquad u_y=0,
\end{equation}
\begin{equation}\label{eq:surface-normal-cubic}
        4A(\partial_z,\partial_z,\partial_z)
        =\mu e^{i\alpha}e^{\lambda z},
\end{equation}
and
\begin{equation}\label{eq:surface-normal-ode}
        u_{xx}
        +\frac{e^{4u}-\mu^2e^{2\lambda x}e^{-4u}}{2}=0
        \quad\text{on }U\cap M^\circ.
\end{equation}
Moreover,
\begin{equation}\label{eq:surface-HCU}
        JH=-2\partial_x.
\end{equation}
These identities hold up to $U\cap \partial M$.
\end{lemma}
\begin{proof}
Choose a smooth boundary isothermal coordinate $\zeta=s+it$ on a
small upper half-disk, with $\zeta(p)=0$ and $g=e^{2v}|d\zeta|^2$.
Write $V=JH=P\partial_s+Q\partial_t$ and set
\[
        \chi=-\frac12(P+iQ)\partial_\zeta
      =g_0(\zeta)\partial_\zeta.
\]
Since $V$ is conformal, $P+iQ$ is holomorphic in the interior. Thus $\chi$ is holomorphic there and smooth up to the diameter, and
\[
\chi+\bar\chi=-\frac12 V.
\] 
Since $H(p)\neq0$, we may shrink the half-disk so that $g_0$ has no zero on its closure. 

Define
\[
        z(\zeta)=\int_0^1
        \frac{\zeta}{g_0(\tau\zeta)}\,d\tau.
\]
The function $z=x+iy$ is smooth up to the diameter and satisfies
\[
z_\zeta=g_0^{-1}, \qquad z_{\bar \zeta}=0
\]
in the interior. These identities extend to the diameter by continuity, so the real Jacobian
is $|g_0|^{-2}>0$ there. 
After extending $z$ smoothly as a real map across the diameter,
the inverse function theorem gives a
one-sided coordinate on a smaller half-disk. 
In this coordinate,
$\chi=\partial_z$, and therefore
\[
JH=-2(\partial_z+\partial_{\bar z})=-2\partial_x.
\]

Write $g=e^{2u}|dz|^2$. Since the Maslov form  is closed and
$\alpha_H=-2e^{2u}dx$, we obtain $u_y=0$ in the interior. Let
$F=4A(\partial_z,\partial_z,\partial_z)$. Since $\Phi$ is holomorphic, so is $F$.

For $e_1=e^{-u}\partial_x$, $e_2=e^{-u}\partial_y$ and
$A_{ijk}=A(e_i,e_j,e_k)$, the trace relation
\eqref{eq:trace-A} and the definition $F$ give
\[
A_{111}+A_{122}=2e^u,\qquad A_{112}+A_{222}=0,
\qquad
2e^{-3u}F=A_{111}-3A_{122}+i(A_{222}-3A_{112}).
\]
The Gauss equation therefore yields
\[
\begin{aligned}
K&=A_{111}A_{122}+A_{112}A_{222}
                 -A_{112}^2-A_{122}^2\\
 &=\tfrac12\bigl(e^{2u}-|F|^2e^{-6u}\bigr).
\end{aligned}
\]
On the other hand, $K=-e^{-2u}u_{xx}$. Consequently,
\begin{equation}\label{eq:surface-Gauss-Codazzi}
u_{xx}+\tfrac12\bigl(e^{4u}-|F|^2e^{-4u}\bigr)=0.
\end{equation}

Differentiating this equation in $y$ gives
$\partial_y|F|^2=0$ at every interior point. If $F$ vanishes at
an interior point, it vanishes on a short vertical line segment
through that point. The holomorphic identity theorem then gives
$F\equiv0$ on the connected neighborhood. In this case, take
$\mu=0$ and $\lambda=\alpha=0$.

Otherwise, $F$ has no interior zeros. Set $q=\log|F|$.
Since $q$ is harmonic and $q_y=0$, we have
\[
        q_{xx}=q_{xy}=0.
\]
Thus $q_x=\lambda$ is constant on the connected neighborhood. The function $q-\lambda x$ also has zero differential, so
\(
        q=\lambda x+\log\mu
\)
for some $\mu>0$.
The holomorphic function $Fe^{-\lambda z}$ has constant modulus $\mu$ and
is therefore a constant $\mu e^{i\alpha}$.
Substitution in \eqref{eq:surface-Gauss-Codazzi} proves
\eqref{eq:surface-normal-cubic} and \eqref{eq:surface-normal-ode}.
All identities extend to the boundary by smoothness.
\end{proof}

\begin{remark}\label{rem:surface-zero-cubic}
If $\mu=0$ in one normal-form chart, then
$\Phi$ vanishes on an interior open set. Since $M^\circ$ is connected, the holomorphic identity theorem gives $\Phi=0$ on all of $M^\circ$. 
As $g(\partial_z,\partial_z)=0$, the $(3,0)$-part of $\mathring A$ equals that of $A$.  A real trace-free
symmetric cubic tensor on a surface is determined by its $(3,0)$-part.
Hence by smoothness,
\begin{equation}\label{eq:surface-tracefree-zero-branch}
        \mathring A=0 \qquad\text{on }M.
\end{equation}
On the boundary, equations \eqref{eq:tracefree-A},
\eqref{eq:surface-boundary-identities}, and
\eqref{eq:surface-W-scalar} give
\begin{equation*}
 0=\mathring A(T,\nu,\nu)
  =A(T,\nu,\nu)+\frac14
       (\alpha_H\mathbin{\odot}g)(T,\nu,\nu)
  =\frac14w.
\end{equation*}
Thus $W=0$ on $\partial M$ in this case.  In the remaining local boundary analysis, we may assume $\mu>0$.
\end{remark}

\subsection{The local boundary system}

Work near a boundary point where $H\neq0$ in the coordinates of Lemma
\ref{lem:surface-normal-form}.  Set
\begin{equation*}
        e_1=e^{-u}\partial_x,
        \qquad
        e_2=e^{-u}\partial_y,
\end{equation*}
and write $A_{ijk}=A(e_i,e_j,e_k)$.  Since
$JH=-2\partial_x=-2e^ue_1$, the trace identity
\eqref{eq:trace-A} gives
\begin{equation}\label{eq:surface-trace-components}
        A_{111}+A_{122}=2e^u,
        \qquad
        A_{112}+A_{222}=0.
\end{equation}
Equation \eqref{eq:surface-normal-cubic} gives
\begin{equation}\label{eq:surface-cubic-components}
 \begin{split}
 A_{111}-3A_{122}
   &=2\mu e^{\lambda x-3u}\cos(\alpha+\lambda y),\\
 A_{222}-3A_{112}
   &=2\mu e^{\lambda x-3u}\sin(\alpha+\lambda y).
 \end{split}
\end{equation}
Solving \eqref{eq:surface-trace-components} and
\eqref{eq:surface-cubic-components}, we obtain
\begin{equation}\label{eq:surface-A-components}
 \begin{aligned}
 A_{111}&=\frac32e^u
          +\frac\mu2e^{\lambda x-3u}\cos(\alpha+\lambda y),\\
 A_{122}&=\frac12e^u
          -\frac\mu2e^{\lambda x-3u}\cos(\alpha+\lambda y),\\
 A_{112}&=-\frac\mu2e^{\lambda x-3u}\sin(\alpha+\lambda y),\\
 A_{222}&=\frac\mu2e^{\lambda x-3u}\sin(\alpha+\lambda y).
 \end{aligned}
\end{equation}

Along the boundary, write
\begin{equation}\label{eq:surface-beta-frame}
 T=\cos\beta\,e_1+\sin\beta\,e_2,
 \qquad
 \nu=-\sin\beta\,e_1+\cos\beta\,e_2.
\end{equation}
Using \eqref{eq:surface-beta-frame} and the components in \eqref{eq:surface-A-components}, we obtain
\begin{equation}\label{eq:surface-ATnn}
 A(T,\nu,\nu)
 =\frac{e^u}{2}\cos\beta
  -\frac\mu2e^{\lambda x-3u}
       \cos(\alpha+\lambda y+3\beta),
\end{equation}
and
\begin{equation}\label{eq:surface-ATTn}
 A(T,T,\nu)
 =-\frac{e^u}{2}\sin\beta
  -\frac\mu2e^{\lambda x-3u}
       \sin(\alpha+\lambda y+3\beta).
\end{equation}
Using \eqref{eq:surface-boundary-identities}, these identities become
\begin{align}
 \mu e^{\lambda x-3u}\cos(\alpha+\lambda y+3\beta)
   &=e^u\cos\beta,
   \label{eq:surface-phase-cos}\\
 \mu e^{\lambda x-3u}\sin(\alpha+\lambda y+3\beta)
   &=-e^u\sin\beta-2\cos\theta.
   \label{eq:surface-phase-sin}
\end{align}
Define the positive function
\begin{equation}\label{eq:surface-r}
        r=\mu^2e^{2\lambda x-8u}>0.
\end{equation}
Then \eqref{eq:surface-normal-ode} becomes
\begin{equation}\label{eq:surface-u-ode}
        u_{xx}+\frac12e^{4u}(1-r)=0.
\end{equation}

If $x$ is non-constant along a boundary arc, choose a smaller arc
$\Gamma$ on which
\begin{equation*}
        T(x)=e^{-u}\cos\beta\neq0.
\end{equation*}
Then $x$ parametrizes $\Gamma$ over an open interval, and
the boundary is a graph $y=\gamma(x)$.  Choose a one-sided coordinate neighborhood over $I$ with connected vertical fibers.  Since $u_y=0$, the function $u$ depends only on $x$, including at the boundary by smoothness. For each $x_0\in I$, equation
\eqref{eq:surface-u-ode} may be evaluated at interior points with
$x=x_0$.  It therefore gives the ordinary differential equation
\begin{equation}\label{eq:surface-u-ordinary}
        u''+\frac12e^{4u}(1-r)=0
        \qquad\text{on }I.
\end{equation}
All boundary identities may now be regarded as identities in $x\in I$.
Primes will denote differentiation with respect to $x$.

\subsection{Boundary rigidity}

We prove that the coordinate $x$ in
Lemma \ref{lem:surface-normal-form} is locally constant on the
boundary wherever $H\neq0$. Otherwise, choose a boundary arc $\Gamma$ as above with
$T(x)\neq0$.
Dots denote differentiation along $T$. 
On $I=x(\Gamma)$, we have $\dot x=e^{-u}\cos\beta\neq0$ and
$y'=\tan\beta$. Combining
\eqref{eq:surface-phase-cos}--\eqref{eq:surface-phase-sin} into one complex
equation gives
\begin{equation}\label{eq:surface-complex-boundary}
\mu e^{\lambda x-3u}e^{i(\alpha+\lambda y+3\beta)}
=e^u e^{-i\beta}-2i\cos\theta.
\end{equation}
Taking absolute values and dividing by $e^{2u}$ yields
\begin{equation}\label{eq:surface-r-boundary}
r=1+4\cos\theta e^{-u}\sin\beta
  +4\cos^2\theta\,e^{-2u}.
\end{equation}

The Levi--Civita connection of
$g=e^{2u(x)}(dx^2+dy^2)$ in the frame
$e_1=e^{-u}\partial_x$, $e_2=e^{-u}\partial_y$ is
\begin{equation}\label{eq:surface-connection}
\begin{aligned}
\nabla_{e_1}e_1&=0,
&\nabla_{e_1}e_2&=0,\\
\nabla_{e_2}e_1&=e^{-u}u'e_2,
&\nabla_{e_2}e_2&=-e^{-u}u'e_1.
\end{aligned}
\end{equation}
Hence
\begin{equation}\label{eq:surface-kg-beta}
k_g=\dot\beta+e^{-u}u'\sin\beta,
\end{equation}
and, since $k_g=-\sin\theta$,
\begin{equation}\label{eq:surface-kg-x}
\cos\beta\,\beta'+u'\sin\beta=-\sin\theta e^u.
\end{equation}

The definition \eqref{eq:surface-r} gives
\begin{equation}\label{eq:surface-basic-two}
\frac{r'}r=2\lambda-8u'.
\end{equation}
Since $\mu>0$, both sides of
\eqref{eq:surface-complex-boundary} are nonzero. Taking its logarithmic
derivative gives
\begin{equation*}
\lambda-3u'+i(\lambda\tan\beta+3\beta')
=
\frac{e^ue^{-i\beta}(u'-i\beta')}
     {e^ue^{-i\beta}-2i\cos\theta}.
\end{equation*}
Using \eqref{eq:surface-r-boundary} and \eqref{eq:surface-kg-x},
the real part becomes
\begin{equation}\label{eq:surface-basic-boundary}
(1+3r)u'-\lambda r-2\sin\theta\cos\theta=0.
\end{equation}
The imaginary part is
\begin{equation*}
\lambda\tan\beta+3\beta'
=
\frac{-\beta'
+2\cos\theta e^{-u}(u'\cos\beta-\beta'\sin\beta)}
{r}.
\end{equation*}
Multiplying by $\cos\beta$, we obtain
\[
r\bigl(\lambda\sin\beta+3\beta'\cos\beta\bigr)
=
-\beta'\cos\beta
+2\cos\theta e^{-u}
 \bigl(u'\cos^2\beta-\beta'\sin\beta\cos\beta\bigr).
\]
By \eqref{eq:surface-kg-x},
$\beta'\cos\beta=-\sin\theta e^u-u'\sin\beta$. Substituting this gives
\[
r(\lambda-3u')\sin\beta-3r\sin\theta e^u
=
\sin\theta e^u+u'\sin\beta
+2\cos\theta e^{-u}u'
+2\sin\theta\cos\theta\sin\beta.
\]
Next, \eqref{eq:surface-basic-boundary} gives
$r(\lambda-3u')=u'-2\sin\theta\cos\theta$.
Substituting this and canceling $u'\sin\beta$ on both sides, we obtain
\[
\sin\theta e^u(1+3r)
+4\sin\theta\cos\theta\sin\beta
+2\cos\theta e^{-u}u'
=0.
\]
Finally, \eqref{eq:surface-r-boundary} gives
$4\cos\theta\sin\beta=e^u(r-1)-4\cos^2\theta\,e^{-u}$.
Hence
\[
4\sin\theta r e^u
+2\cos\theta e^{-u}
 (u'-2\sin\theta\cos\theta)
=0.
\]
Dividing by $2e^u>0$, we arrive at
\begin{equation}\label{eq:surface-phase-derivative}
2\sin\theta r
+\cos\theta e^{-2u}
 (u'-2\sin\theta\cos\theta)
=0.
\end{equation}
No division by $\sin\theta$ or $\cos\theta$ has been used in these
identities.

\subsubsection{The angles $0<\theta<\pi$}

Assume $0<\theta<\pi$. Then $\sin\theta>0$. If $\cos\theta=0$,
\eqref{eq:surface-phase-derivative} gives $2\sin\theta r=0$,
contradicting $r>0$. Hence $\cos\theta\neq0$. Solving
\eqref{eq:surface-basic-boundary} for $u'$ gives
\[
u'=\frac{\lambda r+2\sin\theta\cos\theta}{1+3r},
\]
and hence
$u'-2\sin\theta\cos\theta
=r(\lambda-6\sin\theta\cos\theta)/(1+3r)$.
Substituting this into \eqref{eq:surface-phase-derivative} and using
$r>0$, we obtain
\[
e^{2u}(1+3r)
=
\frac{\cos\theta(6\sin\theta\cos\theta-\lambda)}
     {2\sin\theta}.
\]
The right-hand side is constant, so differentiating with respect to $x$
gives
\[
2u'(1+3r)+3r'=0.
\]
By \eqref{eq:surface-basic-two},
$r'=2r(\lambda-4u')$, while
\eqref{eq:surface-basic-boundary} gives
$\lambda r=(1+3r)u'-2\sin\theta\cos\theta$. Therefore
\[
\begin{aligned}
0
&=2u'(1+3r)+6r(\lambda-4u')\\
&=2u'(1+3r)
  +6\bigl((1+3r)u'-2\sin\theta\cos\theta\bigr)
  -24ru'\\
&=8u'-12\sin\theta\cos\theta.
\end{aligned}
\]
Thus $u'=3\sin\theta\cos\theta/2$. 
\eqref{eq:surface-u-ordinary} then gives $r=1$. The preceding constant
identity forces $u$ to be constant, contradicting
$u'=3\sin\theta\cos\theta/2\neq0$. Consequently,
\begin{equation}\label{eq:surface-x-constant-generic}
x\text{ is locally constant on }\partial M
\quad\text{if }0<\theta<\pi.
\end{equation}

\subsubsection{The endpoint angles}

Assume $\sin\theta=0$, i.e. $\theta\in\{0,\pi\}$. Equation
\eqref{eq:surface-phase-derivative} gives $u'=0$ on $I$. Hence
$u''=0$, and therefore
\[
K=-e^{-2u}u''=0
\]
on $\Gamma$. At $\sin\theta=0$, equations
\eqref{eq:surface-boundary-identities} give
$A(T,\nu,\nu)=0$ and $A(T,T,\nu)=\cos\theta$. The Gauss equation and
the $J\nu$-component of $H$ give
\[
K=\cos\theta\,A(\nu,\nu,\nu)-1,
\qquad
-a=\cos\theta+A(\nu,\nu,\nu)
   =\cos\theta\,(K+2).
\]
Thus
\[
a=-\cos\theta\,(K+2)
\]
on $\partial M$. Equation \eqref{eq:Wa-equations} with $\sin\theta=0$ gives
$T(a)=0$. Hence $K$ is constant on the connected boundary, and since
$K=0$ on $\Gamma$,
\[
K=0\qquad\text{on }\partial M.
\]
Choose $T$ globally with the boundary orientation. For $n=2$, equation
\eqref{eq:gradient-sigma} applied to $V$ gives
\[
T(\rho)=-K\langle V,T\rangle=-Kw=0.
\]
Thus $\rho$ is constant on $\partial M$. The first equation in
\eqref{eq:Wa-equations} gives $T(w)=\rho$. Since $w$ is periodic on the
closed boundary, necessarily $\rho=0$, and hence $w$ is constant.
Finally, $\alpha_H$ is closed by \eqref{eq:closed-conformal}. Stokes'
theorem gives
\[
0=\int_M d\alpha_H
=\int_{\partial M}\alpha_H
=w\,\operatorname{Length}(\partial M),
\]
so $w=0$ on $\partial M$. But on $\Gamma$,
\[
w=\langle -2\partial_x,T\rangle
=-2e^u\cos\beta,
\]
while $\dot x=e^{-u}\cos\beta\neq0$, a contradiction. Therefore
\begin{equation}\label{eq:surface-x-constant-endpoint}
x\text{ is locally constant on }\partial M
\quad\text{if }\theta\in\{0,\pi\}.
\end{equation}

\subsection{Boundary vanishing in dimension two}

We now extract the conclusion in the surface case that is needed in the
proof of Theorem \ref{thm:main}.

\begin{proposition}\label{prop:surface-W-zero}
Under the assumptions of Theorem \ref{thm:main}, if $n=2$, then
\begin{equation}\label{eq:surface-W-zero}
W=0\qquad\text{on }\partial M.
\end{equation}
\end{proposition}

\begin{proof}
If $\Phi\equiv0$, Remark \ref{rem:surface-zero-cubic} gives the result.
Assume that $\Phi\not\equiv0$. Then every normal-form chart centered at a
point where $H\neq0$ has $\mu>0$. Let $p\in\partial M$. If $H(p)=0$,
then $V(p)=JH(p)=0$, so $W(p)=0$. Suppose that $H(p)\neq0$ and use the
coordinate in Lemma \ref{lem:surface-normal-form}. According to the value
of the contact angle, equations \eqref{eq:surface-x-constant-generic}
and \eqref{eq:surface-x-constant-endpoint} show that $x$ is locally
constant along the boundary near $p$. Hence $T(x)=0$. Since
$g=e^{2u}|dz|^2$, the vector $\partial_x$ is orthogonal to $T$ on the
boundary. Equation \eqref{eq:surface-HCU} gives
\begin{equation*}
w=\langle JH,T\rangle
=-2\langle\partial_x,T\rangle
=0.
\end{equation*}
Thus $W(p)=0$. Since $p$ was arbitrary, $W=0$ on all of $\partial M$.
\end{proof}

\section{Boundary rigidity in higher dimensions}

Throughout this section, we assume the hypotheses of Theorem~\ref{thm:main} with $n\geq 3$.
We now prove that $JH$ is normal to the boundary.
Put $\Sigma=\partial M$ and
\begin{equation*}
        m=\dim\Sigma=n-1.
\end{equation*}

\subsection{The trace-free boundary cubic tensor}

Let
\begin{equation}\label{eq:boundary-C}
        C=A|_{T\Sigma^3}.
\end{equation}
If $\{e_\alpha\}_{\alpha=1}^m$ is an orthonormal frame on $\Sigma$,
then \eqref{eq:trace-A} and Lemma \ref{lem:boundary-identities} give
\begin{equation}\label{eq:trace-C}
 \sum_{\alpha=1}^mC(e_\alpha,e_\alpha,U)
 =-\langle W,U\rangle.
\end{equation}
Define
\begin{equation}\label{eq:C-circle}
        \mathring C
        =C+\frac{1}{m+2}(W^\flat\mathbin{\odot}g_\Sigma).
\end{equation}
Equation \eqref{eq:trace-C} shows that $\mathring C$ is trace-free.

We next check its Codazzi property.  For tangent vectors $Z$ and $Q$ on $\Sigma$,
Lemma \ref{lem:boundary-identities} gives
$A(\nu,Z,Q)=\cos\theta\,g(Z,Q)$.  Hence
\eqref{eq:boundary-second-fundamental} gives
\begin{equation}\label{eq:boundary-Codazzi}
 \begin{split}
 (\nabla_XA)(U,Z,Q)
& =(\nabla_X^\Sigma C)(U,Z,Q)
\\& +\sin\theta\cos\theta\bigl\{g(X,U)g(Z,Q)+g(X,Z)g(U,Q)
+g(X,Q)g(U,Z)\bigr\}.
 \end{split}
\end{equation}
The left-hand side and the correction term are symmetric in all four
vectors.  Hence $\nabla^\Sigma C$ is symmetric.  Equation
\eqref{eq:W-conformal} shows that the derivative of
$W^\flat\mathbin{\odot}g_\Sigma$ is also symmetric.  Therefore
\begin{equation}\label{eq:C-circle-Codazzi}
        \nabla^\Sigma\mathring C
        \quad\text{is symmetric}.
\end{equation}

\begin{proposition}\label{prop:boundary-cubic}
If $W$ is not identically zero, then
\begin{equation}\label{eq:C-circle-zero}
        \mathring C=0
\end{equation}
on $\Sigma$.
\end{proposition}

\begin{proof}
By Lemma \ref{lem:harmonic-contraction}, the one-form
\begin{equation*}
        \eta_\Sigma=\mathring C(W,W,\cdot)
\end{equation*}
is closed and coclosed.  Since $m\geq2$ and
$\Sigma\simeq\Sph^m$, there is a function $f_1$ such that
$\eta_\Sigma=df_1$.  The equation $\delta\eta_\Sigma=0$ says that $f_1$ is harmonic.
Since $\Sigma$ is closed, $f_1$ is constant.  Hence
\begin{equation}\label{eq:C-two-W}
        \mathring C(W,W,\cdot)=0.
\end{equation}

The one-form $W^\flat$ is closed by \eqref{eq:W-conformal}. Since
$H^1(\Sigma;\R)=0$, we have $W=\nabla^\Sigma f_2$ for a function $f_2$.
If $W$ is nontrivial, $f_2$ has a maximum and a minimum, so $W$ has a
zero. Lemma \ref{lem:zero-set} shows that the zero set is finite.
Removing finitely many points from $\Sph^m$, $m\geq2$, leaves a
connected set.

We claim that every regular leaf of $W^\perp$ is compact. 
Let $L$ be such a leaf, with $f_2=c$ on $L$.
Then $L$ is a connected component of
$f_2^{-1}(c)\setminus Z(W)$.
If $p\in Z(W)$, then \eqref{eq:W-conformal} gives
\[
        \nabla^2 f_2(p)=\sigma(p)g_\Sigma,
\]
and Lemma \ref{lem:zero-set} gives $\sigma(p)\neq0$. Thus $p$ is a
strict local maximum or minimum of $f_2$. Hence a regular level leaf
cannot accumulate at a zero of $W$. Since $\Sigma$ is compact, the
closure of a regular leaf is compact and contains only regular points;
as a connected component of the corresponding regular level set, the
leaf is closed in that level set. Therefore the leaf itself is compact.

Lemma \ref{lem:round-leaves}, applied on the closed manifold $\Sigma$,
now shows that all regular leaves of $W^\perp$ are round spheres.
Applying Lemma \ref{lem:leaf-reduction} on $U=\Sigma\setminus Z(W)$, with $X=W$ and $S=\mathring C$, and using \eqref{eq:C-two-W}, we obtain
$\mathring C=0$ on $U$. The result follows everywhere by
continuity.
\end{proof}

\subsection{Two boundary Ricci identities}

We now prove the main boundary result.

\begin{proposition}\label{prop:W-zero}
Under the assumptions of Theorem \ref{thm:main}, if $n\geq3$, then
\begin{equation}\label{eq:W-zero}
        W=0\qquad\text{on }\Sigma.
\end{equation}
\end{proposition}

\begin{proof}
Suppose that $W$ is not identically zero. 
Recall that $a=\alpha_H(\nu)$.
Taking the $J\nu$ component of $H$ and using
Lemma \ref{lem:boundary-identities}, we get
\begin{equation}\label{eq:boundary-A-nnn}
 -a=\langle H,J\nu\rangle
 =\sum_{\alpha=1}^mA(e_\alpha,e_\alpha,\nu)
   +A(\nu,\nu,\nu)
 =m\cos\theta+A(\nu,\nu,\nu).
\end{equation}

Recall from \eqref{eq:curvature-X} that
\begin{equation}\label{eq:integrability-V}
        R(U,\nu)V=U(\rho)\nu-\nu(\rho)U.
\end{equation}
Taking the inner product with $\nu$ and using the Gauss equation on
$M^n$, we obtain
\begin{equation}\label{eq:rho-first}
 \begin{split}
 U(\rho)
 & =\langle h(U,\nu),h(\nu,V)\rangle
    -\langle h(U,V),h(\nu,\nu)\rangle.
 \end{split}
\end{equation}
Lemma \ref{lem:boundary-identities} gives
\begin{equation*}
 \langle h(U,\nu),h(\nu,V)\rangle
 =\cos\theta A(\nu,V,U)
 =\cos^2\theta\langle W,U\rangle.
\end{equation*}
Since $\{Je_1,\ldots,Je_m,J\nu\}$ is an orthonormal normal frame and
$A(\nu,\nu,e_\alpha)=0$, it also gives
$h(\nu,\nu)=-(a+m\cos\theta)J\nu$ by
\eqref{eq:boundary-A-nnn}. Therefore
\begin{equation*}
 \langle h(U,V),h(\nu,\nu)\rangle
 =-(a+m\cos\theta)A(U,V,\nu)
 =-\cos\theta(a+m\cos\theta)\langle W,U\rangle.
\end{equation*}
Using the preceding identities and $n=m+1$, we obtain
\begin{equation}\label{eq:rho-gradient}
        \nabla^\Sigma\rho=\cos\theta(a+n\cos\theta)W.
\end{equation}
Equations \eqref{eq:Wa-equations} and \eqref{eq:boundary-sigma} give
\begin{equation}\label{eq:sigma-gradient}
        \nabla^\Sigma\sigma
        =\bigl[\cos\theta(a+n\cos\theta)-\sin^2\theta\bigr]W.
\end{equation}

Using \eqref{eq:W-conformal}, we have
\(
R^\Sigma(X,Y)W
=X(\sigma)Y-Y(\sigma)X,
 \) for \(X,Y\in T\Sigma.\)
Taking the trace, we obtain
\begin{equation}\label{eq:Ric-first}
 \Ric^\Sigma(W,U)=-(m-1)U(\sigma).
\end{equation}
Consequently,
\begin{equation}\label{eq:Ric-first-W}
 \Ric^\Sigma(W,W)
 =-(m-1)\bigl[\cos\theta(a+n\cos\theta)-\sin^2\theta\bigr]|W|^2.
\end{equation}

We compute the same Ricci curvature from the immersion of $\Sigma$
in the unit sphere.  If $\bar h$ is the second fundamental form of
$\Sigma\subset\mathbb S^{2n-1}$, then
$\langle\bar h(U,Z),JN\rangle
=-\langle Z,JU\rangle=0$.  Hence
\begin{equation}\label{eq:sphere-second-form}
\begin{aligned}
\bar h(U,Z)
&=\sum_{\alpha=1}^m
  \langle\bar h(U,Z),Je_\alpha\rangle Je_\alpha\\
&=\sum_{\alpha=1}^m C(U,Z,e_\alpha)Je_\alpha.
\end{aligned}
\end{equation}
By \eqref{eq:trace-C}, the trace of $\bar h$ is $-JW$.
The Gauss equation in the unit sphere gives
\begin{equation}\label{eq:Ric-C}
 \Ric^\Sigma(W,W)
 =(m-1)|W|^2-C(W,W,W)-|C(W,\cdot,\cdot)|^2.
\end{equation}

Proposition \ref{prop:boundary-cubic} gives
\begin{equation}\label{eq:C-W-formula}
 C(W,X,Y)
 =-\frac{1}{m+2}
 \bigl\{|W|^2g(X,Y)
       +2\langle W,X\rangle\langle W,Y\rangle\bigr\}.
\end{equation}
Hence
\begin{equation}\label{eq:C-W-values}
 C(W,W,W)=-\frac{3}{m+2}|W|^4,
 \qquad
 |C(W,\cdot,\cdot)|^2
 =\frac{m+8}{(m+2)^2}|W|^4.
\end{equation}
Substituting them into \eqref{eq:Ric-C} gives
\begin{equation}\label{eq:Ric-second-W}
 \Ric^\Sigma(W,W)
 =(m-1)|W|^2
 +\frac{2(m-1)}{(m+2)^2}|W|^4.
\end{equation}

Comparison of \eqref{eq:Ric-first-W} and
\eqref{eq:Ric-second-W}, together with $m+2=n+1$ and
$\sin^2\theta+\cos^2\theta=1$, gives
\begin{equation}\label{eq:Ric-comparison-factor}
 0=(m-1)|W|^2\left\{
 \cos\theta\bigl[a+(n+1)\cos\theta\bigr]
 +\frac{2}{(n+1)^2}|W|^2\right\}.
\end{equation}
Here $m\geq2$.  Thus, on the open set where $W\neq0$, we may divide
by $(m-1)|W|^2$ and obtain
\begin{equation}\label{eq:key-boundary-algebra}
 \cos\theta\bigl[a+(n+1)\cos\theta\bigr]
 +\frac{2}{(n+1)^2}|W|^2=0.
\end{equation}
The regular set is dense by Lemma \ref{lem:zero-set}, so
\eqref{eq:key-boundary-algebra} holds on all of $\Sigma$ by
continuity.

If $\cos\theta=0$, equation \eqref{eq:key-boundary-algebra} gives $W=0$, a
contradiction.  Suppose that $\sin\theta>0$ and $\cos\theta\neq0$.  By
\eqref{eq:Wa-equations},
\begin{equation*}
        da=\sin\theta W^\flat.
\end{equation*}
At both a maximum point and a minimum point of $a$, one has $W=0$.
Equation \eqref{eq:key-boundary-algebra} gives
\begin{equation*}
        a=-(n+1)\cos\theta
\end{equation*}
at both points.  Thus the maximum and minimum of $a$ agree, so $a$ is
constant and $W=0$, again a contradiction.

It remains to consider $\sin\theta=0$. Then
$\cos\theta=\pm1$.  Equation \eqref{eq:Wa-equations} shows that $a$ is constant,
and \eqref{eq:key-boundary-algebra} shows that $|W|$ is constant.
The vector field $W$ is the gradient of a function on the compact
manifold $\Sigma$ (see the proof of Proposition \ref{prop:boundary-cubic}), so it has a zero.  Therefore $|W|=0$, which is again a contradiction.
\end{proof}
\section{Proof of Theorem \ref{thm:main}}

We apply the boundary vanishing propositions and then use the
trace-free cubic tensor to classify the immersion.

Proposition \ref{prop:surface-W-zero} gives the boundary vanishing result when
$n=2$, and Proposition \ref{prop:W-zero} gives it when $n\geq3$.
Therefore, 
\begin{equation}\label{eq:V-normal-boundary}
V=JH=a\nu\qquad\text{on }\Sigma.
\end{equation}
The second equation in \eqref{eq:Wa-equations}, together with $W=0$ and the connectedness of $\Sigma$, shows that $a$ is
constant on $\Sigma$.

We separate the minimal and nonminimal cases.

\subsection{The minimal case}

If $H\equiv0$, then $\Delta|\iota|^2=2n>0$.
Since $|\iota|^2=1$ on the boundary and $|\iota|^2<1$ in the interior,
the Hopf boundary point lemma gives
\[
2\sin\theta=\partial_\nu|\iota|^2>0.
\]
Thus $\theta\in(0,\pi)$, and the immersion is self-similar with parameter zero.
The boundary is connected, so its constant contact angle has a
cosine of fixed sign. Corollary~3.8 of \cite{GaoMaYao2026} applies
and shows that $\iota$ is a diffeomorphism onto an equatorial
Lagrangian disk.

\subsection{The nonminimal case}

Assume now that $H$ is not identically zero.  Since $\alpha_H$ is
closed and $H^1(M^n;\R)=0$, there is a function $f$ such that
\begin{equation}\label{eq:V-gradient}
        V=\nabla f.
\end{equation}
Equation \eqref{eq:V-normal-boundary} shows that $f$ is constant on
$\Sigma$.  The boundary constant $a$ cannot be zero.  Indeed, if
$a=0$, then $V=0$ on $\Sigma$.  Its tangential derivative also
vanishes there, so \eqref{eq:closed-conformal} gives $\rho=0$ on
$\Sigma$.  Lemma \ref{lem:zero-set} would then give $V\equiv0$, which
contradicts the present case.

The function $f$ is nonconstant and has constant boundary value.
Therefore either its maximum or its minimum is attained in the
interior, and $V$ has an interior zero. Lemma \ref{lem:zero-set}
shows that its zeros are isolated. Since $V$ has no boundary zeros,
the zero set is finite. Removing finitely many points from the
interior of a ball does not disconnect it, so
\begin{equation}\label{eq:regular-connected}
        (M^n)^\circ\setminus Z(V)\quad\text{is connected}.
\end{equation}

Every regular leaf of $V^\perp$ in $(M^n)^\circ$ is compact. Let $L$ be such a leaf, with $f=c$ on $L$.
Then $L$ is a connected component of
$f^{-1}(c)\cap((M^n)^\circ\setminus Z(V))$.  At an interior zero $p$ of
$V$, equation \eqref{eq:closed-conformal} gives
\[
        \nabla^2f(p)=\rho(p)g,
\]
and Lemma \ref{lem:zero-set} gives $\rho(p)\neq0$. Thus $p$ is a
strict local maximum or minimum of $f$, so the leaf cannot accumulate
at $Z(V)$. It cannot accumulate at the boundary either. Indeed, $f$ is
constant on $\Sigma$ and $df(\nu)=a\neq0$ there. In an inward collar of
the compact boundary, the normal derivative therefore has a fixed
nonzero sign after the collar is chosen sufficiently small. Hence the
boundary level of $f$ has no interior points in that collar; all other
levels are separated from the boundary by continuity. Thus the closure
of $L$ in $M^n$ is contained in $f^{-1}(c)\cap((M^n)^\circ\setminus Z(V))$. As a connected component, $L$ is
closed in this set. Hence $L$ is closed in the compact manifold $M^n$ and is compact. It is therefore a compact hypersurface without boundary contained in $(M^n)^\circ$.

Consider the trace-free cubic tensor $\mathring A$ defined in
\eqref{eq:tracefree-A}.  Lemma \ref{lem:harmonic-contraction} shows
that
\begin{equation}\label{eq:zeta}
        \eta_M=\mathring A(V,V,\cdot)
\end{equation}
is closed and coclosed.  Since $H^1(M^n;\R)=0$, write
\begin{equation*}
        \eta_M=d\varphi.
\end{equation*}
For $U\in T\Sigma$, equations \eqref{eq:V-normal-boundary} and
Lemma \ref{lem:boundary-identities} give
\begin{equation}\label{eq:zeta-boundary}
 \begin{split}
 \eta_M(U)
 &=a^2\mathring A(\nu,\nu,U)\\
 &=a^2A(\nu,\nu,U)=0.
 \end{split}
\end{equation}
Thus $\varphi$ is constant on the connected boundary.  Since
$\delta\eta_M=0$, the function $\varphi$ is harmonic.  The Dirichlet
maximum principle yields
\begin{equation}\label{eq:zeta-zero}
        \eta_M=0.
\end{equation}

Applying Lemma \ref{lem:round-leaves} to $V$ on $(M^n)^\circ$, using
\eqref{eq:regular-connected} and the compactness just proved, shows
that the regular leaves of $V^\perp$ are round spheres. Applying Lemma \ref{lem:leaf-reduction} on $U=(M^n)^{\circ}\setminus Z(V)$, with $X=V$ and $S=\mathring A$, and using \eqref{eq:zeta-zero}, we obtain
\begin{equation}\label{eq:A-circle-zero}
        \mathring A=0
\end{equation}
on $(M^n)^\circ\setminus Z(V)$ and then on all of $M^n$ by continuity.
Equations \eqref{eq:tracefree-A} and \eqref{eq:trace-A} show that
\eqref{eq:A-circle-zero} is equivalent to
\begin{equation}\label{eq:Whitney-second-form}
 h(X,Y)=\frac{1}{n+2}
 \bigl\{g(X,Y)H+\langle JX,H\rangle JY
                    +\langle JY,H\rangle JX\bigr\}.
\end{equation}
By \cite[Theorem~2]{RosUrbano1998}, the immersion is totally geodesic
or its image is part of a Whitney sphere. Since $H$ is not
identically zero, after a unitary change of coordinates fixing the origin its image is contained in $W_{r,c}(\Sph^n)$ for some $r>0$ and $c\in \C^n$. This change of coordinates preserves the unit ball and the Legendrian capillary boundary condition. It remains to prove that $c=0$ in these coordinates.

Away from the two poles, write the Whitney immersion as
\[
W_{r,c}(t,p)=c+z(t)p,\qquad
z(t)=\frac{r\sqrt{1-t^2}(1+it)}{1+t^2},
\qquad p\in\Sph^{n-1}.
\]
The Lagrangian phase of this parametrization is
$\pi/2+(n+2)\arctan t$, up to an additive constant. Indeed,
the complex determinant of a tangent frame has the same argument
as $z'(t)z(t)^{n-1}$. The Maslov form is the differential of the
phase up to sign \cite[Section~2]{RosUrbano1998}. Thus $JH$ is
orthogonal to each latitude.
The induced metric satisfies
\[
g(\partial_t,\partial_t)
=\frac{r^2}{(1-t^2)(1+t^2)}.
\]
Together with the phase formula, this gives
\[
|JH|^2
=\frac{(n+2)^2(1-t^2)}{r^2(1+t^2)}.
\]
By smoothness, $JH$ vanishes at the two poles.
The boundary has no zeros of $JH$, because $a\neq0$ in
\eqref{eq:V-normal-boundary}. Hence $t$ is well-defined near
$\iota(\Sigma)$, and $V=a\nu$ gives $U(t)=0$ for every
$U\in T\Sigma$. Connectedness of $\Sigma$ shows that $t=t_0$
on its image. The projection of $\Sigma$ to $p\in\Sph^{n-1}$ is
a local diffeomorphism. Its image is both open and closed, and
therefore it is all of $\Sph^{n-1}$.

Put $\zeta=z(t_0)\neq0$. Because the boundary lies in the unit
sphere, for every $p\in\Sph^{n-1}$ we have
\[
1=|c+\zeta p|^2
 =|c|^2+|\zeta|^2+2\operatorname{Re}(\bar\zeta c)\cdot p.
\]
This identity implies $\operatorname{Re}(\bar\zeta c)=0$.
The Legendrian condition gives, for every $U\perp p$,
\[
0=\langle J(c+\zeta p),\zeta U\rangle
 =-\operatorname{Im}(\bar\zeta c)\cdot U.
\]
Since $p$ ranges over the whole sphere,
$\operatorname{Im}(\bar\zeta c)=0$ as well. Thus $c=0$.
Finally, $|\zeta|=1$ gives
$r^2(1-t_0^2)/(1+t_0^2)=1$, so $r\geq1$ and
$t_0^2=(r^2-1)/(r^2+1)$. This proves the theorem.

\section*{Acknowledgments}
This work was supported by the National Natural Science Foundation of China (Grant Nos. 12671062, 12201138, 12401057, 12471048, W2521103), the Natural Science Foundation of Henan Province (Grant No. 262300421869) and the Beijing Natural Science Foundation (Grant No. 1244039).

\end{document}